\documentclass[11pt]{amsart}
\usepackage{hyperref}
\usepackage{amsfonts,amssymb,amsmath,tikz,amsthm,amscd,euscript,array,mathrsfs,comment}
\usepackage[all]{xy}
\usepackage{bbm}
\usepackage{booktabs}
\usepackage{color}
\usepackage{mathdots}
\usetikzlibrary{matrix}
\usepackage{enumitem}
\input{mymacros-3.sty} 

\title{Radial restriction of spherical functions on supergroups}
\author{Mitra Mansouri \and Hadi Salmasian}
\address{Department of Mathematics and Statistics,
University of Ottawa,
STEM Complex,
150 Louis-Pasteur Pvt,
Ottawa, ON,
Canada K1N 6N5}
\email{hsalmasi@uottawa.ca}
\address{Department of Mathematics and Statistics,
University of Ottawa,
STEM Complex,
150 Louis-Pasteur Pvt,
Ottawa, ON,
Canada K1N 6N5}
\email{mmans053@uottawa.ca}
\date{\today}

\begin{document}

\begin{abstract}
Using the Hopf superalgebra structure of the enveloping algebra $U(\g g)$ of a Lie superalgebra $\g g=\mathrm{Lie}(G)$, we  give a purely algebraic treatment of $K$-bi-invariant functions on a Lie supergroup $G$, where $K$ is a sub-supergroup of $G$. 
We realize $K$-bi-invariant functions as a  subalgebra $\cA(\g g,\g k)$ of the dual of $U(\g g)$ whose elements vanish on the coideal $\cI=\g kU(\g g)+U(\g g)\g k$, where $\g k=\mathrm{Lie}(K)$.  
Next, for a general class of supersymmetric pairs $(\g g,\g k)$, we define the radial restriction of elements of $\cA(\g g,\g k)$ and  prove that it is an injection into $S(\g a)^*$, where $\g a$ is the Cartan subspace of $(\g g,\g k)$.  
Finally, we compute a basis for $\cI$ in the case of the pair
$(\g{gl}(1|2), \g{osp}(1|2))$, and uncover a connection with the Bernoulli and Euler zigzag numbers.

\end{abstract}

\maketitle

\section{Introduction}

A spherical function on a real reductive Lie group $G$ is a  $K$-bi-invariant
eigenfunction of the algebra of invariant differential operators on $G/K$, where $K$ is a maximal compact subgroup of $G$. The theory of spherical functions is extensive and has profound connections with  representation theory and combinatorics of symmetric functions \cite{helgason2022groups,Macdonald,Gangolli}. 
However, in the context of  supergroups this theory is far less developed. This is partly due to the cumbersome technicalities that arise in general in superanalysis and supergeometry, and create  major hurdles in extending the  tools of harmonic analysis to the super setting. 

To circumvent the above technical difficulty, we pursue an indirect, purely algebraic approach to study spherical functions on supersymmetric spaces.  This approach was first explored by Sergeev \cite{sergeev2001superanalogs} for two families of symmetric pairs $(\g{gl}(m|n)\oplus \g{gl}(m|n),\g{gl}(m|n))$ and $(\g{gl}(m|2n),\g{osp}(m|2n))$. 
An extended, yet unpublished, version of~\cite{ sergeev2001superanalogs}
is available in 
\cite{SergeevPreprint}.
However, Sergeev addressed each of the above families by a different and explicit method. From this viewpoint, the papers   
\cite{sergeev2001superanalogs,SergeevPreprint} use ad hoc methods.

In this paper, we expand on the idea of Sergeev in the general setting of supersymmetric pairs $(\g g,\g k)$ of Lie superalgebras. Let $\cA(\g g,\g k)$ denote the dual of the quotient
$U(\g g)/\cI$ where \[
\cI:=\g kU(\g g)+U(\g g)\g k.
\] Since $\cI$ is a coideal, $\cA(\g g,\g k)$ is a (super)commutative algebra. Elements of $\cA(\g g,\g k)$ play the role of $K$-bi-invariant functions on $G$. Of course the Lie algebra $\g g$ does not contain the information about isogeny of $G$.  But  
in some sense $\cA(\g g,\g k)$ contains all analytic spherical functions on all isogenous pairs $(G,K)$ associated to $(\g g,\g k)$, in the same way that $U(\g g)^*$ contains all analytic functions  
on the isogeny class of $G$
 in the non-super setting (where we identify functions by their derivatives at  identity). 
Our main result is Theorem~\ref{injection}, which states that $K$-bi-invariant functions are uniquely determined by their radial restriction. This is a variant of a well-known fact 
from the purely even case (see Remark~\ref{rmk:GK}), and leads to interesting problems that we aim to explore in the future. 
Theorem~\ref{injection} was also proved in~\cite[Theorem 4.1]{SergeevVeselovSym} for several families of pairs $(\g g,\g k)$, under the assumption that $\g g$ has a non-degenerate invariant form that restricts to a non-degenerate form on the Cartan subspace $\g a$ of $\g g$. Our proof has a small advantage over the one given in~\cite{SergeevVeselovSym}: it does not use the invariant form of $\g g$, and therefore it applies to a larger class of pairs $(\g g,\g k)$; see Example~\ref{exa-q}.

We conclude this paper by looking more closely at  the supersymmetric pair $(\mathfrak{gl}(1|2), \mathfrak{osp}(1|2))$. We compute a basis of $\mathcal{I}=\g k U(\g g)+U(\g g)\g k$ in Theorem \ref{basis of I}. The interesting observation is that  Bernoulli  numbers occur as coefficients in our basis. Along the way, we also come across the Euler zigzag numbers. \\

\noindent\textbf{Acknowledgement.} We thank the anonymous referee for drawing our attention to the work of Sergeev and Veselov in~\cite{SergeevVeselovSym}. This paper is based on the doctoral dissertation of the first author, which was completed under supervision of the second author. Certain aspects of this project were refined through  discussions with 
the participants of the AIM SQuaREs meetings on 
Symmetric spaces and Capelli operators for Lie superalgebras.  
 The research of the second author is partially supported by NSERC Discovery Grant RGPIN-2024-04030.

\section{The algebra $\cA(\g g,\g k)$}\label{chapter1}

Throughout this paper, the base field will be $\mathbb{C}$.
Let $\g g$ be any Lie superalgebra. Then the universal enveloping algebra $U(\g g)$ is a Hopf superalgebra, with co-product $\Delta:U(\g g)\to U(\g g)\otimes U(\g g)$ defined by $\Delta(x):=1\otimes x+x\otimes 1$ for $x\in \g g$, antipode defined by $S(x)=-x$ for $x\in \g g$, and counit defined by $\eps(1)=1$ and $\eps(x)=0$ for $x\in \g g$.
Using the comultiplication of $U(\g g)$, we can equip the dual  $U(\g g)^*$ with an associative algebra structure. The multiplication of $U(\g g)^*$ is given by
\[
\phi\psi(x):=\sum (-1)^{|x_1||\psi|} \phi(x_1)\psi(x_2)\text{ for $x\in U(\g g)$,}
\] 
where $\Delta(x)=\sum x_1\otimes x_2$ and as usual $|\cdot|$ denotes  parity.

\begin{rmk}
Let $H$ be a Hopf (super)algebra. An element of $H^*$ is said to be of \emph{finite type} if its kernel contains a two-sided ideal of $H$ of finite codimension. The elements of finite type of $H^*$ form a Hopf
algebra (see~\cite[Ch. VI]{sweedler1969hopf}), which we will denote by $H^\circ$.  
\end{rmk}

Let $(\pi,V)$ be a finite dimensional $\g g$-module. Recall that the matrix coefficients of $(\pi,V)$ are linear maps $\phi_{v,v^*}\in U(\g g)^*$ defined by 
\begin{equation}
\label{phivv*eq}
\phi_{v,v^*}(x):=\lag v^*,\pi(x)v\rag
\text{ for all }x\in U(\g g),\ v\in V,\ v^*\in V^*.
\end{equation}
If $\mathrm{Ann}(V)$ denotes the annihilator
of a $\g g$-module module $V$, then $U(\g g)/\mathrm{Ann}(V)$ embeds in 
$\mathrm{End}_{\mathbb{C}}(V)$, hence
 $\phi_{v,v^*}\in U(\g g)^\circ$. Indeed the converse also holds and is straightforward to prove:  elements of $U(\g g)^\circ$ are  
 matrix coefficients of finite dimensional $\g g$-modules.


Let $\g k\sseq \g g$ be any subalgebra of $\g g$ and set $\cI:=\g kU(\g g)+U(\g g)\g k$. We say $\la\in U(\g g)^*$ is  \emph{$\g k$-bi-invariant}  if 
$
\cI\sseq\ker(\lambda)
$.
Since $\g k$ is $\Z_2$-graded, the $\g k$-bi-invariant $\lambda\in U(\g g)^*$ form a graded subspace which we denote by $\cA(\g g ,\g k)$.

\begin{lem}
The superspace $\cA(\g g,\g k)$ is a subalgebra of $U(\g g)^*$. 
\end{lem}
\begin{proof}
This is immediate from the fact that $\cI$ is a coideal of $U(\g g)$, but we provide more details. Given $ \lambda, \mu \in \mathcal A(\g g,\g k) $, we want to prove that $ \lambda\mu\in \mathcal A(\g g,\g k) $. For $ k\in \g k $ and $ u\in U(\g g) $, if $\Delta(u)=\sum u_1\otimes u_2$ in Sweedler's notation then 
\begin{align*}
\lambda \mu (ku)=\lambda\otimes \mu ( (&k\otimes 1+1\otimes k)\Delta(u) )=\lambda\otimes \mu \left( \left(k\otimes 1+1\otimes k\right)\sum u_1 \otimes u_2\right)\\&=\sum \left(
(-1)^{|\mu|\cdot(|k|+|u_1|)}
\lambda \left(ku_1\right)\otimes \mu \left(u_2\right)+(-1)^{(|k|+|\mu|) |u_1|
}\lambda \left(u_1\right)\otimes \mu \left(ku_2\right)\right)=0,
\end{align*}
hence $\g kU(\g g)\sseq \ker(\lambda\mu)$. Similarly, 
$U(\g g)\g k\sseq \ker (\lambda\mu)$, so that 
$\lambda\mu\in \mathcal A(\g g,\g k)$. 
\end{proof}
The superalgebra $U(\g g)$ acts by left and right translation on itself  as follows:
\begin{equation}
\label{eq:LxRxdf} L_x:y\mapsto xy \quad \text{and}\quad R_x:y\mapsto (-1)^{|x|.|y|}yx.
\end{equation}
Now set 
\[
\cA(\g g,\g k)^\circ:=\cA(\g g, \g k)\cap U(\g g)^\circ.
\] 
The justification that $\cA(\g g,\g k)^\circ$ plays the role of spherical functions for $(\g g,\g k)$ is the following proposition. 
\begin{prp}\label{rols-of-spherical}
	Let $\la\in 
U(\g g)^\circ$. Then $\lambda\in	
	\cA(\g g,\g k)^\circ$ if and only if  
	$\la(x)=\phi_{v,v^*}$ for $v\in V$ and $v^*\in V^*$, where $V$ is a finite dimensional $\g g$-module, $V^*$ is the dual of $V$, and both $v$ and $v^*$ are $\g k$-fixed. 
\end{prp}

\begin{proof}
If $\lambda=\phi_{v,v^*}$ for $\g  k$-fixed $v$ and $v^*$, then checking that $\lambda\in\cA(\g g,\g k)$ is straightforward. Conversely, 
	let $\la\in \cA(\g g,\g k)^\circ$. Since the sum of matrix coefficients on two $\g g$-modules $V_1$ and $V_2$ is a matrix coefficient on $V_1\oplus V_2$,  we can assume that $\lambda$ is homogeneous. Suppose that  $\la(I)=0$ for a two-sided ideal of finite codimension. In the rest of this proof we set $U:=U(\g g)$.
	The contragredient of the left translation on $U^*$ is given by
	\begin{equation}\label{dual-action}
	L_x^*\mu(y):=(-1)^{|x||\mu|}\mu(L_{S(x)}y)=(-1)^{|x||\mu|}\mu(S(x)y), \quad\text{ for }x\in U(\g g),
	\end{equation}
	where $S:U\to U$ is the antipode. 
	Now set $W:=\spn\{L_x^*(\la\circ S): x\in U \}\sseq U^*$. For homogeneous elements $u\in I$ and $x\in U$ we have
	\[
	L_u^*(\la\circ S)(x)=(-1)^{|u||x|+|u||\la|}\la(S(x)u)=0,
	\]
	because $S(x)u\in I$. Therefore the kernel of the linear map $ U(\g g)\rightarrow W$, $ x \mapsto L_x^* (\lambda \circ S)$, contains $I$. Hence, $\dim W\leq \dim U/I<\infty$. Set
	$W^{\perp}=\cap_{\mu\in\mu}\ker(\mu)$.
	Now, consider the evaluation map
	\[ \ev:U(\g g)\to W^*,\quad u\mapsto \ev_u\]
	where
	$\ev_u(\mu):=\mu(u)$. Since $W$ is finite dimensional, the above map is a surjection onto $W^*$ and it follows that $U/W^\perp\cong W^*$. We can now obtain $\la$ as a matrix coefficient for the representation $(\rho,W)$ where  $\rho(x)=(-1)^{|x|\cdot|\lambda|}L_x^*$. 
By \eqref{dual-action} above, we have
	$
	\la(x)=(-1)^{|x||\la|}\ev_1(L_x^*(\la\circ S))
	$.
	The vector $\la\circ S$ is $\g k$-invariant because for any $k \in \mathfrak k$ we have:
	\[
	L_k^*(\la\circ S)(y)=(-1)^{|k||\la|}\la(S(y)k)=0.
	\]
	The dual vector $1=\ev_1$ is also $\g k$-invariant because from $S(\g k)\sseq \g k$ it follows that
	\[k\cdot \ev_1(L_x^*(\la\circ S))=(-1)^{|k||\la|}
	L_x^*(\la\circ S(k))=(\la\circ S)(S(x)k)=(-1)^{|k||x|}\la (S(k)x)=0. 
\qedhere \] \end{proof}

\section{Radial restriction and injectivity for symmetric pairs}
\label{sec:Radialres}
Let $\theta$ be an involution of $\g g$. Then $\g g=\g k\oplus\g p$ where $\g k$ and $\g p$ are the $\pm 1$ eigenspaces of $\theta$.
Consider the super-symmetrization map
$\sfs:S(\g p)\to U(\g g)$ given by
\[
x_1\cdots x_r\in S^r(\g p)\mapsto 
\frac{1}{r!}\sum_{\pi\in S_r}
\sgn(\pi;x)x_{\pi(1)}\cdots x_{\pi(r)},
\] 
where $\sgn(\pi;x)$ denotes the sign of the permutation that is induced by $\pi$ on the odd vectors $x_{i_1},x_{i_2},\dots , x_{i_s}$ among $x_1,\dots , x_r\in\g p$. 
In what follows we set $\cI_L:=U(\g g)\g k$.
\begin{prp}\label{prp:S(p)IL}
	The natural map
	$
	\oline\sfs:S(\g p)\to U(\g g)/\cI_L
	$
	that is induced by  $\sfs:S(\g p)\to U(\g g)$ is an isomorphism of vector superspaces. 
\end{prp}
\begin{proof}
Set $S^{\leq d}(\g p):=\bigoplus_{k=0}^d S^k(\g p)$. To prove surjectivity of $\oline \sfs$, we prove by induction on $d$ that for any monomial $p_1\cdots p_d\in U(\g g)$ with homogeneous $p_i\in\g p$ we have  $\cI_L+p_1\cdots p_d\in \oline \sfs(S^{\leq d}(\g p))$. 
This is clear for $d\leq 1$. For a given $d>1$, first note that
\begin{align}
\label{eq:ppsps|}
p_1\cdots p_jp_{j+1}\cdots p_d=
(-1)^{|p_j|\cdot|p_{j+1}|} p_1\cdots p_{j+1}p_j\cdots p_d+
p_1\cdots p_{j-1}qp_{j+2}\cdots p_d
,\end{align}
 where $q=[p_j,p_{j+1}]$. 
 Furthermore, for any homogeneous $p'_1,\ldots,p'_{d-1}\in\g p$, $q'\in \g k$, and  $1\leq j'\leq d-2$ we have
 \[
 p'_1\cdots p'_{j'}q'p'_{j'+1}\cdots p'_{d-1}=
(-1)^{|q'|\cdot |p'_{j'+1}|} 
  p'_1\cdots p'_{j'}p'_{j'+1}q'p'_{j'+2}\cdots p'_{d-1}+p'_1\cdots p'_{j'}p'p'_{j'+2}\cdots p'_{d-1},
 \]
where $p':=[q',p'_{j'+1}]\in \g p$ is homogeneous. By the induction hypothesis, the second term on the right hand side is in $\sfs(S^{\leq (d-1)}(\g p))$. Thus, by repeating the above process of moving $q$ to the right, it follows that 
$p_1\cdots p_{j-1}qp_{j+2}\cdots p_d\in \cI_L+\sfs(S^{\leq (d-1)}(\g p))$. Consequently, from~\eqref{eq:ppsps|} we obtain that for any permutation $\pi$ of $1,\ldots,d$ we have
\[
p_1\cdots p_d\in
\sgn(\pi;p)p_{\pi(1)}\cdots p_{\pi(d)}+\cI_L+\sfs(S^{\leq (d-1)}(\g p)).
\]
Averaging on both sides over all permutations $\pi$ yields
\begin{equation}
\label{eq:p1pdd}
p_1\cdots p_d\in \sfs(p_1\cdots p_d)+\cI_L+
\sfs(S^{\leq (d-1)}(\g p))\sseq\cI_L+
\sfs(S^{\leq d}(\g p)).
\end{equation}
This proves surjectivity of $\oline \sfs$.

For injectivity, we choose a homogeneous basis $p_1,\ldots,p_m$ for $\g p$. From the PBW theorem for $\g g$ it follows that the monomials $p_1^{r_1}\cdots p_m^{r_m}$ represent a basis for $U(\g g)/\cI_L$. 
From~\eqref{eq:p1pdd} it follows that 
\[
\sfs(p_1^{r_1}\cdots p_m^{r_m})\in p_1^{r_1}\cdots p_m^{r_m}+\oline\sfs(S^{\leq(d-1)}(\g p))+\cI_L.
\]
But the latter relation also implies that elements of 
$\sfs(S^{\leq(d-1)}(\g p))$ can also be expressed, modulo $\cI_L$, as a linear combination of monomials $p_1^{s_1}\cdots p_m^{s_m}$ where $\sum_{i=1}^m s_i\leq d-1$. Thus, if $A_d$ denotes the matrix whose columns contain the weights obtained from expressing the 
$\oline \sfs(p_1^{r_1}\cdots p_m^{r_m})$ for $\sum_{i=1}^m r_i\leq d$ as linear combinations of the basis vectors $p_1^{s_1}\cdots p_m^{s_m}+\cI_L$ of $U(\g g)/\cI_L$ for $\sum_{i=1}^ms_i\leq d$, then $A_d$ is  unitriangular, hence invertible. This proves injectivity of $\oline \sfs$. 
\end{proof}
\begin{cor}
The map
	$
	\oline\sfs:S(\g p)\to U(\g g)/\cI_L
	$
 dualizes to an isomorphism of vector superspaces 
	\[
	S(\g p)^*\cong (U(\g g)/\cI_L)^*.
	\] 
\end{cor}

\begin{lem}\label{kS(p)inI}
$\sfs(\ad_\g k(S(\g p)))\sseq\cI$.
\end{lem}
\begin{proof}
The map $\sfs$ is the restriction of the supersymmetrization map $ S(\g g)\to U(\g g)$, which is a $\g g$-module homomorphism. It follows that $\oline\sfs$ is a $\g k$-module homomorphism. Thus for homogeneous $a\in\g k$ and $y\in  S(\g p)$ we have 
\[
\sfs(\ad_a(y))=\ad_a( \sfs(y))=a\sfs(y)-\sfs(y)a\in\cI.
\qedhere\]
\end{proof}

\begin{lem}\label{Zariski-dense}
Let $V$ be a finite dimensional complex vector space and let $S\sseq V$ be Zariski dense. Then $S^n(V)$ is spanned by $\{a^n\,:\,a\in S\}$.
\end{lem}
\begin{proof}
Let $f:S^n(V)\to\C$ be a linear functional such that $f(a^n)=0$ for all $a\in S$. Choose a basis $e_1,\ldots,e_d$ for $V$. Then  $S^n(V)$ has a basis of the form $e_1^{m_1}\cdots e_d^{m_d}$, where the $m_i\geq 0$ and $\sum_i m_i=n$. 
Set $c_{m_1,\ldots,m_d}:=f(e_1^{m_1}\cdots e_d^{m_d})$. Then for $a\in S$ of the form $a=a_1e_1+\cdots+a_de_d$ where the $a_i$ are in $\C$, we have
\begin{align*}
	0=f(a^n)&=
	\sum_{m_1+\cdots +m_d=n}
	{m_1,\ldots,m_d \choose n}c_{m_1,\ldots,m_d}
	a_1^{m_1}\cdots a_d^{m_d}.
	\end{align*}
The right hand side is a polynomial in $a_1,\ldots,a_d$, and since $S$ is Zariski dense it follows that $c_{m_1,\ldots,m_d}=0$ for all $d$-tuples $(m_1,\ldots,m_d)$. 
	Hence $f=0$, and this proves the claim. 
\end{proof}
From now on, we assume that we can  choose a subalgebra $\g h\sseq \g g_\eev$ with the following properties:
\begin{itemize}
\item[(i)] $\theta(\g h)=\g h$, so that 
$\g h=\g t\oplus\g a$ where $\g t:=\g h\cap\g k$ and $\g a:=\g h\cap\g p$. 
\item[(ii)] The family of operators $\{\ad_x\}_{x\in\g h}\sseq\End(\g g)$ is simultaneously diagonalizable.

\item[(iii)] $C_\g  g(\g h)=\g h$ and $C_{\g p}(\g a)=\g a$. Here  $C_{\g b}(\g a)=\{x\in\g b\,:\,[x,y]=0\text{ for all }y\in\g a\}$.

\end{itemize}
Let \[
\g g=\g h\oplus\bigoplus_{\alpha\in \Delta}\g g_\alpha
\] be the root space decomposition of $\g g$ with respect to  $\g h$, where $\Delta\sseq \g h^*\bls\{0\}$ and
\[
\g g_{\alpha}=\lbrace x\in \g g \,|\,{\ad}_{h}(x)=\alpha(h)x\text{ for all }h\in \g h \rbrace.
\]
We set $\theta(\alpha):=\alpha\circ\theta$ for $\alpha\in\Delta$. Note that $\theta(\alpha)\in\Delta$. 
In Lemma~\ref{adatothei} below, for $a\in\g a$ we consider $a^k$ as an element of $S^k(\g p)$. Furthermore, the adjoint action of $\g k$ on $\g p$ equips  $S^k(\g p)$ with a $\g k$-module structure. 
\begin{lem}\label{adatothei}
Let $x\in \g g_{\alpha}$ where $\alpha\in \Delta$. Then $
	\ad_{x+\theta(x)}(a^k)
	=-k\alpha(a)(x-\theta(x))a^{k-1}$ for $a\in\g a$ and $k\geq 1$. 
\end{lem}
\begin{proof}
 We use induction on $k$. For $k=1$, from $(\theta(\alpha))(a)=-\alpha(a)$ it follows that
\[ \ad_{x+\theta(x)}(a)=\ad_{x}(a)+\ad_{\theta(x)}(a)=-\alpha(a)x+\alpha(a)\theta(x)=-\alpha(a)(x-\theta(x)).
\] 
Assuming the assertion holds for $k-1$, we have
\begin{align*}
\ad_{(x+\theta(x))}(a^k)&=(\ad_{(x+\theta(x))}(a))a^{k-1}+a(\ad_{(x+\theta(x))}(a^{k-1}))\\
&=(-\alpha(a)(x-\theta(x))a^{k-1})+(-(k-1)\alpha(a)a(x-\theta(x))a^{k-2}).
\end{align*}
Since $S(\g p)$ is commutative and $x-\theta(x)\in \g p$, we have $a(x-\theta(x))=(x-\theta(x))a$. Thus the right hand side of the latter formula is equal to $-k\alpha(a)(x-\theta(x))a^{k-1}$.
\end{proof}
For $\alpha \in \Delta$ we set $\g e_\alpha:=\g g_\alpha+\g g_{\theta(\alpha)}$. Then we have the following lemma.
\begin{lem}\label{alpha-is-nonzero}
 If $\g e_\alpha\cap \g p\neq \lbrace 0\rbrace$, then $\alpha|_{\g a}\neq 0$.
\end{lem}
\begin{proof}
Suppose that $\alpha|_{\g a}= 0$. Then $\alpha\circ \theta|_{\g a}=-\alpha|_{\g a}=0$ and hence $[\g e_\alpha,\g a]=0$. In particular,  $\g e_\alpha \cap \g p$ commutes with $\g a$. This contradicts the assumption that $C_{\g p}(\g a)=\g a$.
\end{proof}

The following statement is proved by Lepowsky \cite{lepowsky1975harish} in the case of Lie algebras. We follow Lepowsky's argument with minor modifications.

\begin{prp}\label{U=S(a)}
	$S(\g p)=S(\g a)+\ad_{\g k}(S(\g p))$.
\end{prp} 
\begin{proof}
We prove a stronger statement: if $W$ is the smallest $\ad_\g k$-invariant subspace of $S^r(\g p)$ that contains $S^r(\g a)$, then $W=S^r(\g p)$. 
	Note that 
	\[
	W=S^r(\g a)+\sum_{x\in \g k}\ad_xS^r(\g a)+
	\sum_{x,x'\in \g k}\ad_x\ad_{x'}S^r(\g a)+\cdots 
	\sseq S^r(\g a)+
	\ad_{\g k} S^r(\g p).
	\]
	Thus, if we prove that $W=S^r(\g p)$, then the proposition follows.  

Given $\alpha\in\Delta$, for $x\in\g g_\alpha$ we have $\theta(x)\in\g g_{\theta(\alpha)}$.
 Thus, $\theta(\g e_\alpha)=\g e_\alpha$ and consequently $\g e_\alpha=(\g e_\alpha\cap \g k)\oplus (\g e_{\alpha}\cap\g p)$. In addition, 
	\[
	\g e_\alpha\cap\g k=\{x+\theta(x)\,:\,x\in\g g_\alpha
	\}\quad
	\text{and}\quad 
	\g e_\alpha\cap\g p=\{x-\theta(x)\,:\,x\in\g g_\alpha\}. 
	\]
	It follows that we have a vector superspace decomposition $\g p=\g a\oplus \g q$ where $\g q:=\bigoplus_{\alpha\in\Delta}(\g e_\alpha\cap \g p)$. Next note that
	\[
	S^r(\g p)=\bigoplus_{i=0}^r S^{r-i}(\g q)S^i(\g a).
	\]
	We prove  that $S^{r-i}(\g q)S^i(\g a)\sseq W$ for $0\leq i\leq r$, by a reverse induction on $i$.
	The statement is trivial for 
	$i=r$. Next assume that $i\geq 1$ and the claimed inclusion holds for all $i'\geq i$. Since \[
	S^{r-(i-1)}(\g q)=S^{r-i}(\g q)S^1(\g q)=S^{r-i}(\g q)\g q,
	\] 
	it suffices to prove that  
	$qea^{i-1}\in W$ for $q\in S^{r-i}(\g q)$, $a\in \g a$, and $e\in \g q$. Since elements of $\g q$ are linear combinations of elements of $\g e_\alpha\cap \g p$, it suffices to assume that $e\in \g e_\alpha\cap\g p$ for some $\alpha$. Thus, we can assume that $e=x-\theta(x)$ for some $x\in \g g_\alpha$. By choosing $e$ to be homogeneous with respect to the $\mathbb{Z}_2$-grading, we can assume $x$ is also homogeneous. Note that we can also assume that $q$ is homogeneous.

Since $e\in \g e_\alpha\cap\g p$ and we can assume $e\neq 0$, from
Lemma \ref{alpha-is-nonzero} it follows that $\alpha|_{\g a}\neq 0$.	Recall that  $e=x-\theta(x)$ where $x\in\g g_\alpha$. By Lemma \ref{adatothei}, 
	\begin{align*}
	\ad_{x+\theta(x)}(qa^i)&=(\ad_{x+\theta(x)}q)a^i
	+(-1)^{|q||x|}q\ad_{(x+\theta(x))}(a^i)\\
	&=(\ad_{x+\theta(x)}q)a^i
	-(-1)^{|q||x|}i\alpha(a)q(x-\theta(x))a^{i-1}.
	\end{align*}
	By the induction hypothesis, the first term on the right hand side belongs to $W$. Also, the left hand side belongs to $W$ (by $\ad_\g k$-invariance of $W$ and the induction hypothesis). It follows that
	\[
	\alpha(a)q(x-\theta(x))a^{i-1}\in W.
	\]
Since $\alpha\big|_{\g a}\neq 0$, there exists  a Zariski dense subset $S\sseq \g a$ such that  $\alpha(a)\neq 0$ for $a\in S$. Thus if $a\in S$ then 
	\[
	q(x-\theta(x))a^{i-1}\in W.
	\]
If $i=1$ then the proof is complete. Next assume that $i>1$. 	Since the choices of $e=x-\theta(x)$ form a spanning set of $\g q$, from the above arguments it follows that
	\[
	(S^{r-i}(\g q)\g q)a^{i-1}\sseq W\text{ if }
	a\in S.
	\]
Finally, note that Lemma \ref{Zariski-dense} implies that  $S^{i-1}(\g a)$ is spanned by 
$\{a^{i-1}\,:\,a\in S\}$. 
\end{proof}
Since $\cI_L\sseq \cI$, 
we have a natural quotient map
 $\mathsf{q}:U(\g g)/\cI_L\to U(\g g)/\cI$. In Corollary~\ref{radial-part}, we consider $\bar\sfs$ as a map with domain $S(\g a)\sseq S(\g p)$. 
\begin{cor}\label{radial-part}
	$U(\g g)=\cI+\sfs(S(\g a))$. In particular the map 
	$\mathsf q\circ \oline \sfs:S(\g a)\to U(\g g)/\cI$ is a surjection.
\end{cor}
\begin{proof}
By Proposition \ref{prp:S(p)IL} we have $
\sfs(S(\g p))+\cI_L=U(\g g)$, hence by Proposition \ref{U=S(a)} we obtain
\[\sfs(S(\g a))+\sfs({\ad}_{\g k}S(\g p))+\cI_L=U(\g g).
\] 
But by Lemma \ref{kS(p)inI}, we have  $\sfs({\ad}_{\g k}S(\g p))\subseteq \mathcal{I}$, hence  $\oline\sfs(S(\g a))+\cI=U(\g g)$.
Surjectivity of $\mathsf q\circ\oline\sfs$ is a trivial consequence of the latter fact. 
\end{proof}

In Theorem~\ref{injection} below, note that
$S(\g a)$ is canonically isomorphic to $U(\g a)$, hence it is a Hopf algebra. 

\begin{thm}\label{injection}
	The map 
	$(\mathsf q\circ\oline\sfs)^*:\cA(\g g,\g k)\to S(\g a)^*$   is an  embedding of commutative algebras. It restricts to an embedding $\cA(\g g,\g k)^\circ\to S(\g a)^\circ$. 
\end{thm}

\begin{proof}
Injectivity of $(\mathsf q\circ\oline\sfs)^*$ follows from Corollary~\ref{radial-part}. 
Since $\g a$ is abelian, the map $\sfs: S(\g a)\to U(\g g)$ is the canonical embedding of Hopf algebras. The coalgebra  structures of $U(\g g)$ and $S(\g a)$ induce the algebra structures of $\cA(\g g,\g k)$ and $S(\g a)^*$, and $(\mathsf q\circ\oline\sfs)^*$ is compatible with these structures. Finally, if $\lambda\in \cA(\g g,\g k)^\circ$ then $\dim(S(\g a)/\ker(\lambda)\cap S(\g a))\leq \dim(U(\g g)/U(\g g)\cap\ker(\lambda))<\infty$, hence $(\mathsf q\circ\oline\sfs)^*(\lambda)\in S(\g a)^\circ$. 
\end{proof}

\begin{rmk}
\label{rmk:GK}Let $G$ be a real reductive Lie group with a Cartan decomposition $G=KAK$ where $K$ is the maximal compact subgroup of $G$. Then the restriction of a $K$-bi-invariant function on $G$ to $A$ is invariant under the action of the little Weyl group associated to $A$. 
It is natural to expect that, analogously,  the image  of 
the map $(\mathsf q\circ\oline\sfs)^*$ of Theorem~\ref{injection} can be   characterized by symmetry conditions imposed by the Weyl groupoid of the restricted root system of $(\g g,\g k)$. 
 We summarize the results in the literature that address the latter problem. In the following discussion, we assume that  the restricted root system of $(\g g,\g k)$ is
of type $A_\kappa(r,s)$, as defined in~\cite[Sec. 2]{SVdeformedq}. In particular, the restriction of the invariant form of $\g g$ to $\g a$ is non-degenerate.  The canonical action of the
little Weyl group $W=W(\g a)$ of $\g g_\eev$ on  $\g a^*$ induces an action on $S(\g a)^*$ where, as usual, to transfer the action we identify $\g a$ and $\g a^*$ using the bilinear form on $\g a$. We choose a basis $\mathscr B:=\{\eps_i\}_{i=1}^{r+1}\cup\{\delta_{\bar j}\}_{j=1}^{s+1}$
for $\g a^*$ such that the roots of  $A_\kappa(r,s)$ can be  represented as 
\[
\{\eps_i-\eps_{i'}\}_{1\leq i\neq i'\leq r+1}
\cup
\{\delta_{\bar j}-\delta_{\bar j'}\}_{1\leq j\neq j'\leq s+1}
\cup\{
\pm(\eps_i-\delta_{\bar j})
\}_{1\leq i\leq r+1\,,\,1\leq j\leq s+1 }\,,
\]
and we have $\kappa=-\mathrm{sdim}(\eps_i-\eps_{i'})/2=-2/{\mathrm{sdim}(\delta_{\bar j}-\delta_{{\bar j}'}})$, where $\mathrm{sdim}(\alpha):=\dim(\g g_\alpha)_\eev-\dim(\g g_\alpha)_\ood$.   Let $\{h_i\}_{i=1}^{r+1}\cup \{h_{\bar j}\}_{j=1}^{s+1}$
be the basis of $\g a$ that is dual to $\mathscr B$. 
Let $S(\g a)^*_\mathrm{Inv}$ be the subspace of $S(\g a)^*$ containing those $\lambda\in S(\g a)^*$  that are $W$-invariant and satisfy the ``supersymmetry'' constraint \[
\lambda((h_1-\kappa h_{\bar 1})(h_1+h_{\bar 1})^N)=0
,\] for all $N\geq 0$.
In~\cite[Prop. 5.7]{SergeevVeselovSym}, it is proved that for $(\g g,\g k)=(\gl(m|2n),\g{osp}(m|2n))$, if $\lambda\in \cA(\g g,\g k)$  is of the form $\lambda=\phi_{v,v^*}$, as in~\eqref{phivv*eq}, for $\g k$-fixed vectors $v\in V$ an $v^*\in V^*$, where $V$ is an irreducible $\g g$-module, then
\[
(\mathsf q\circ\oline\sfs)^*(\lambda)\in S(\g a)^*_\mathrm{Inv}
.
\] 
In~\cite{MitraThesis}, it is proved that for the pairs \[
(\g{gl}(m|n)\oplus\g{gl}(m|n),\g{gl}(m|n))\ ,\ 
(\g{gl}(m|2n),\g{osp}(m|2n))
\ ,\   
(\g{gosp}(m|2n),\g{osp}(m-1,2n)),
\] we have
$
(\mathsf q\circ\oline\sfs)^*(\cA(\g g,\g k))\sseq S(\g a)^*_\mathrm{Inv}
$.
The values of the parameter $\kappa$ that correspond to the above 3 pairs are $\kappa=-1,-1/2,n-(m-1)/2$.

The above constraints on the image of 
$(\mathsf q\circ\oline\sfs)^*(\cA(\g g,\g k))$ have the following more concrete interpretation. To any $\lambda\in U(\g g)^*$ we can associate the formal power series 
\[
\Phi_\lambda:=\sum_{a,b}\lambda(h_1^{a_1}\cdots h_{r+1}^{a_{r+1}}h_{\bar 1}^{b_1}\cdots h_{\oline{s+1}}^{b_{s+1}})
\frac{\eps_1^{a_1}}{a_1!}
\cdots
\frac{\eps_{r+1}^{a_{r+1}}}{a_{r+1}!}
\frac{\delta_{\bar 1}^{b_1}}{b_1!}
\cdots
\frac{\delta_{\oline{s+1}^{}}^{b_{s+1}}}{b_{s+1}!},
\]
where the summation is on all tuples $a=(a_1,\ldots,a_{r+1})$ and $b=(b_1,\ldots,b_{s+1})$ of non-negative integers. 
Note that we can evaluate the formal series $\Phi_\lambda$ at points in $\g a$, whenever the resulting power series  converges.  Then $W$-invariance of $\lambda$ is tantamount to $\Phi_\lambda$ being symmetric separately in the $\eps_i$ and in the $\delta_{\bar j}$. For $x\in U(\g g)$,  we define $D_x\Phi_\lambda:=\Phi_{R_x^*(\lambda)}$, where $R_x^*:U(\g g)^*\to U(\g g)^*$ is dual to the right translation action $R_x$ of~\eqref{eq:LxRxdf}, i.e., $R_x^*(\lambda)(y)=(-1)^{|\lambda|\cdot |x|+|x|\cdot |y|}\lambda(yx)$. Set $\alpha:=\eps_1-\delta_{\bar 1}$ and let $h_\alpha:=h_1-\kappa h_{\bar 1}\in\g a$ denote the corresponding coroot. The supersymmetry constraint states that $D_{h_\alpha}\Phi_\lambda$ vanishes on the hyperplane $\alpha=0$, where $D_{h_\alpha}=\frac{\partial}{\partial \eps_1}-\kappa\frac{\partial}{\partial \delta_{\bar 1}}$.

Finally, we remark that if $\Omega$ denotes  the Casimir operator of $\g g$, then  
  \begin{equation*}
\begin{split}
D_\Omega&= -\kappa \sum_{\substack{1 \leqslant i,j \leqslant r+1 \\ i \neq j}} \frac{\eps_i}{\eps_i - \eps_j}\left( \eps_i \frac{\partial}{\partial \eps_i} - \eps_j \frac{\partial}{\partial \eps_j} \right) 
- \sum_{\substack{1 \leqslant i,j \leqslant s+1 \\ i \neq j}} \frac{\delta_{\bar{i}}}{\delta_{\bar{i}} - \delta_{\bar{j}}}\left( \delta_{\bar{i}} \frac{\partial}{\partial \delta_{\bar{i}}} - \delta_{\bar{j}} \frac{\partial}{\partial \delta_{\bar{j}}} \right) \\
&\quad - \sum_{\substack{1 \leqslant i \leqslant r+1 \\ 1 \leqslant j \leqslant s+1}} \frac{\eps_i + \delta_{\bar{j}}}{\eps_i - \delta_{\bar{j}}}\left( \eps_i \frac{\partial}{\partial \eps_i}  -\kappa \delta_{\bar{j}} \frac{\partial}{\partial \delta_{\bar{j}}} \right)
+ \sum_{1 \leqslant i \leqslant r+1} \left( \eps_i \frac{\partial}{\partial \eps_i} \right)^2  
+ \kappa \sum_{1 \leqslant i \leqslant s+1} \left( \delta_{\bar{i}} \frac{\partial}{\partial \delta_{\bar{i}}} \right)^2.
\end{split}
\end{equation*}
This is the Calogero-Moser-Sutherland operator, which is the Hamiltonian of the integrable system corresponding to the one-dimensional quantum $n$-body problem. See~\cite{SergeevVeselovSym,MitraThesis} for further details.
\end{rmk} 

\begin{ex}
\label{exa-q}
Let $\g g:=\gl(n|n)$ and let $\theta:\g g\to \g g$ be defined by 
\[
\theta\left(
\begin{bmatrix}
A& B\\
C& D\end{bmatrix}\right)
:=\begin{bmatrix}
D & C \\
B & A\end{bmatrix}.
\]
Then $\g k=\g{q}(n)$. Now let $\g h$ be the standard diagonal Cartan subalgebra of $\g g$. Then $\g a$ is the subspace of diagonal matrices of the form
\[
\mathrm{diag}(a_1,\ldots,a_n,-a_1,\ldots,-a_n),\quad a_1,\ldots,a_n\in\C. 
\]
It is straightforward to verify that  $C_\g p(\g a)=\g a$.
The invariant form of $\g g$ (which is $(x,y):=\mathrm{str}(xy)$) vanishes on $\g a$. Thus, $(\gl(n|n),\g{q}(n))$ is covered by Theoerem~\ref{injection},
but not by~\cite[Theorem 4.1]{SergeevVeselovSym}. 
\end{ex}

\section{A basis for $\cI$  for the pair $(\gl(1|2),\g{osp}(1|2))$}
We begin this section by a quick review of Euler zigzag numbers. By an alternating permutation of $\{1,\ldots,n\}$ we mean a bijection $\sigma$ of this set such that for every $1\leq i\leq n-2$ we have $\sigma(i)<\sigma(i+1)$ if and only if $\sigma(i+1)>\sigma(i+2)$. 
Let $A_n$ for $n\in\Z^{\geq 0} $ denote the number of alternating permuations (we set $A_0:=1)$.
By  a result of 
Andr\'{e} 
\cite{andre1881permutations},  
\[
\tan(x) + \sec(x) = \sum_{n \geqslant 0} \frac{A_n}{n!} x^n.
\]
Note that the $A_{n}$ for odd $n$ are the coefficients of the Taylor series of $\tan(x)$ and we have 
\[
A_{2m-1}=\frac{(-1)^{m-1}2^{2m}(2^{2m}-1)}{2m}B_{2m}\quad\text{for }m\geq 1,
\]
where  $\{B_n\}_{n\geq 0}$ is the sequence of Bernoulli numbers. The numbers \[
E_{2n}:=(-1)^nA_{2n}
\] occur in the coefficients of $\mathrm{sech}(x)$ and are known as  \emph{Euler zigzag numbers}  
 \cite{graham1994concrete,abramowitz1968handbook}. 
 It is usually assumed that $E_{2n+1}=0$.

Henceforth let $(\g g,\g k):=(\gl(1|2),\g{osp}(1|2))$. Set $I_{\bar{0}}:=\{1\}$,  $I_{\bar{1}}:=\{\bar{1},\bar{2}\}$,
and $I:=I_{\bar 0}\cup I_{\bar 1}$.  Let $\{E_{i,j}\, |\, i,j\in I\}$ be the basis of $\gl(1|2)$ that consists of standard matrix units. We define  $|j|:=0$ if $j\in I_{\bar{0}} $, and $|j|:=1$ if $j\in I_{\bar{1}} $.
The involution corresponding to $(\g g,\g k)$ is $\theta(X):=-PX^\mathrm{st}P^{-1}$ where $X^\mathrm{st}$ is the supertranspose of $X$ and 
\begin{align*}
P:=\begin{bmatrix}
1&0&0\\
0&0&1\\
0&-1&0
\end{bmatrix}.
\end{align*}
Then
\begin{align*}
\g k=\left\lbrace\begin{bmatrix}
0&a_{12}&a_{13}\\
-a_{13}&a_{22}&a_{23}\\
a_{12}&a_{32}&-a_{22}
\end{bmatrix}\ \Bigg|\ 
a_{12},a_{13},a_{22},a_{23},a_{32}\in \mathbb{C}\right\rbrace,
\end{align*}
and
\begin{align*}
\g p=
\left\lbrace\begin{bmatrix}
a_{11}&a_{12}&a_{13}\\
a_{13}&a_{22}&0\\
-a_{12}&0&a_{22}
\end{bmatrix} \ \Bigg| \  a_{11},a_{12},a_{13},a_{22}\in \mathbb{C}\right\rbrace.
\end{align*}Indeed $\g k_{\bar{0}}\cong \g{sl}(2,\mathbb{C})$ and we have
\[
\g p_{\bar{0}}=\left\lbrace\begin{bmatrix}
a&0&0\\
0&b&0\\
0&0&b
\end{bmatrix} \ \Bigg| \ a,b\in \mathbb{C}\right\rbrace\quad\text{and}\quad \g p_{\bar{1}}=\left\lbrace\begin{bmatrix}
0&a&b\\
b&0&0\\
-a&0&0
\end{bmatrix} \ \bigg| \ a,b\in \mathbb{C}\right\rbrace.
\]
We choose $\g h$ to be the diagonal Cartan subalgebra of $\g g$. Then $\g a=\spn\{h_1,h_{\bar 1}\}$ where $h_1=E_{1,1}$ and $h_{\bar 1}=E_{\bar 1,\bar 1}+E_{\bar 2,\bar 2}$. 
Our next goal is to choose a basis for $\mathfrak{gl}(1|2)$ with convenient commutator relations that is compatible with the decomposition $\g g=\g k\oplus\g p$. To this end, for the rest of this section we make the following choices:
\[
p:=\begin{bmatrix}
2&0&0\\
0&1&0\\
0&0&1
\end{bmatrix}\in \g p_{\bar{0}},\quad e:=\begin{bmatrix}
0&1&0\\
0&0&0\\
-1&0&0
\end{bmatrix}\in \g p_{\bar{1}},\quad f:=\begin{bmatrix}
0&0&1\\
1&0&0\\
0&0&0
\end{bmatrix}\in \g p_{\bar{1}}.
\]
\[k:=\begin{bmatrix}
0&0&0\\
0&1&0\\
0&0&-1
\end{bmatrix}\in \g k_{\bar{0}},\quad
k_1:=\begin{bmatrix}
0&0&0\\
0&0&1\\
0&0&0
\end{bmatrix}\in \g k_{\bar{0}},\quad k_2:=\begin{bmatrix}
0&0&0\\
0&0&0\\
0&1&0
\end{bmatrix}\in \g k_{\bar{0}}.
\]
\[
z:=
\begin{bmatrix}
1&0&0\\
0&1&0\\
0&0&1
\end{bmatrix}\in\g p_\eev,
\quad e':=\begin{bmatrix}
0&1&0\\
0&0&0\\
1&0&0
\end{bmatrix}\in \g k_{\bar{1}},\quad f':=\begin{bmatrix}
0&0&1\\
-1&0&0\\
0&0&0
\end{bmatrix}\in \g k_{\bar{1}},
\]
Then $\{z,k,k_1,k_2,e',f',p,e,f\}$ is  a basis of $\g g$ and by the  PBW theorem the monomials 
\[z^tk^{r_1}k_1^{r_2}k_2^{r_3}(e')^{r_4}(f')^{r_5}p^{s_1}e^{s_2}f^{s_3}\quad\text{where}
\quad
t,r_1,r_2,r_3,s_1\in\Z^{\geq 0}\text{ and }\ 
r_4,r_5,s_2,s_3\in\{0,1\}
\]
form a basis of $U(\g g)$.  

For a vector $v=(a,b)\in\C^2$ we define
\[
v_{\mathfrak p}:=\begin{bmatrix}
0&a&b\\
b&0&0\\
-a&0&0
\end{bmatrix}\in \g p_{\bar{1}}\quad\text{and}\quad v_{\mathfrak k}:=\begin{bmatrix}
0&a&b\\
-b&0&0\\
a&0&0
\end{bmatrix}\in \g k_{\bar{1}}.
\]
It is straightforward to verify that  $[v_{\mathfrak p},p]=-v_{\mathfrak k}$ and $[v_{\mathfrak k},p]=-v_{\mathfrak p}$.

\begin{lem}\label{supersymmetrygl(1,2)}
For $n\in\mathbb Z^{\geq 0}$ we have  
\begin{align*}
p^nv_{\mathfrak k}=v_{\mathfrak k} \sum_{\tiny
\begin{array}{cc}
0\leqslant i\leqslant n\\
i \, \text{is even}
\end{array}}{n \choose i}p^{n-i}+v_{\mathfrak p} \sum_{\tiny
\begin{array}{cc}
0\leqslant i\leqslant n\\
i \, \text{is odd}
\end{array}}{n \choose i}p^{n-i}.
\end{align*}
Furthermore,  $v_{\mathfrak p}p^m\in \mathcal{I}$ for all  $0\leq m\leq n$.
\end{lem}
\begin{proof}
We have $
pv_{\mathfrak k}=v_{\mathfrak k}p-[v_{\mathfrak k},p]=v_{\mathfrak k}p+v_{\mathfrak p}$, hence  $v_{\mathfrak p}\in \mathcal{I}$.
As $v\in \C^2$ is arbitrary, we obtain $\g p_{\bar{1}}\subseteq \mathcal{I}$.
Now suppose the assertions of the lemma hold for $n-1$. First, we assume that $n$ is odd (the argument for $n$ even  is analogous). Then 
\begin{align*}
p^{n}v_{\mathfrak k}
=p(p^{n-1}v_{\mathfrak k})
&=pv_{\mathfrak k} \sum_{\tiny
\begin{array}{cc}
0\leqslant i\leqslant n-1\\
i \, \text{is even}
\end{array}}{n-1 \choose i}p^{n-1-i}+pv_{\mathfrak p} \sum_{\tiny
\begin{array}{cc}
0\leqslant i\leqslant n-1\\
i \, \text{is odd}
\end{array}}{n-1 \choose i}p^{n-1-i}\\
&=
(v_{\mathfrak k}p+v_{\mathfrak p}) \sum_{\tiny
\begin{array}{cc}
0\leqslant i\leqslant n-1\\
i \, \text{is even}
\end{array}}{n-1 \choose i}p^{n-1-i}+(v_{\mathfrak p}p+v_{\mathfrak k}) \sum_{\tiny
\begin{array}{cc}
0\leqslant i\leqslant n-1\\
i \, \text{is odd}
\end{array}}{{n-1} \choose i}p^{n-1-i}\\
&=
v_{\mathfrak k}\left(
\sum_{\tiny
\begin{array}{cc}
0\leqslant i\leqslant n-1\\
i \, \text{is even}
\end{array}}{{n-1} \choose i}p^{n-i}+\sum_{\tiny
\begin{array}{cc}
0\leqslant i\leqslant n-1\\
i \, \text{is odd}
\end{array}}{n-1 \choose i}p^{n-1-i}\right)\\
&+
v_{\mathfrak p}\left(\sum_{\tiny
\begin{array}{cc}
0\leqslant i\leqslant n-1\\
i \, \text{is even}
\end{array}}{{n-1} \choose i}p^{n-1-i}+\sum_{\tiny
\begin{array}{cc}
0\leqslant i\leqslant n-1\\
i \, \text{is odd}
\end{array}}{n-1 \choose i}p^{n-i}\right).
\end{align*}
For the first term on the right hand side we have 
\begin{align*}
\sum_{\tiny
\begin{array}{cc}
0\leqslant i\leqslant n-1\\
i \, \text{is even}
\end{array}}{{n-1} \choose i}p^{n-i}
&+\sum_{\tiny
\begin{array}{cc}
0\leqslant i\leqslant n-1\\
i \, \text{is odd}
\end{array}}{n-1 \choose i}p^{n-1-i}\\
&=p^n+\sum_{
\tiny
\begin{array}{cc}
2\leqslant i\leqslant n-1\\
i \, \text{is even}
\end{array}}
\left({n-1 \choose i}+{n-1 \choose i-1}\right)p^{n-i}\\
&=p^n+\sum_{
\tiny
\begin{array}{cc}
2\leqslant i\leqslant n-1\\
i \, \text{is even}
\end{array}}
{n \choose i}p^{n-i}=\sum_{\tiny
\begin{array}{cc}
0\leqslant i\leqslant n-1\\
i \, \text{is even}
\end{array}}{n \choose i}p^{n-i}.
\end{align*}
By a similar argument, for the second term we obtain
\begin{align*}
\sum_{\tiny
\begin{array}{cc}
0\leqslant i\leqslant n-1\\
i \, \text{is odd}
\end{array}}{{n-1} \choose i}p^{n-1-i}+\sum_{\tiny
\begin{array}{cc}
0\leqslant i\leqslant n-1\\
i \, \text{is even}
\end{array}}{n-1 \choose i}p^{n-i}=\sum_{\tiny
\begin{array}{cc}
0\leqslant i\leqslant n\\
i \, \text{is odd}
\end{array}}{{n} \choose i}p^{n-i}.
\end{align*}
Consequently, 
\begin{align*}
p^nv_{\mathfrak k}=v_{\mathfrak k} \sum_{\tiny
\begin{array}{cc}
0\leqslant i\leqslant n\\
i \, \text{is even}
\end{array}}{n \choose i}p^{n-i}+v_{\mathfrak p} \sum_{\tiny
\begin{array}{cc}
0\leqslant i\leqslant n\\
i \, \text{is odd}
\end{array}}{n \choose i}p^{n-i},
\end{align*}
which proves the first assertion. 
From isolating $v_{\g p}p^n$ on the right hand side and using the induction hypothesis that  $v_{\mathfrak p} p^{m} \in \mathcal{I}$ for  $m \leqslant n-1$, it follows that  $v_{\mathfrak p} p^{n} \in \mathcal{I}$ as well.
\end{proof}
The commutator relations between $v_{\g p}$, $v_{\g k}$ and $p$ imply that for every $n\in\Z^{\geq 0}$ there exist unique polynomials $\alpha_n(x)$ and $\beta_n(x)$, independent of the choice of $v$, such that 
\begin{equation}
\label{eq:dfnab}
p^nv_{\mathfrak k}=v_{\mathfrak k}\alpha_{n}(p)+\beta_{n-1}(p)v_{\mathfrak p}.
\end{equation}
\begin{lem}\label{alpha(p)}
The following assertions hold.
\begin{itemize}
\item[\rm (i)] $\deg(\alpha_n)=n$, $\deg(\beta_{n-1})=n-1$, and the leading coefficients of $\alpha_n$ and $\beta_{n-1}$ are positive. 
\item[\rm (ii)]
$v_{\mathfrak p}p^n=\alpha_{n}(p)v_{\mathfrak p}-v_{\mathfrak k}\beta_{n-1}(p)$.
\item[\rm (iii)]
$\alpha_{n}(x)=x\alpha_{n-1}(x)-\sum_{i=1}^{n-1}a_{i}\beta_{i-1}(x)$ and $\beta_{n-1}(x)=x\beta_{n-2}(x)+\sum_{i=0}^{n-1}a_i\alpha_{i}(x)$,
where we assume $\alpha_{n-1}(x)=a_0+a_1x+\cdots +a_{n-1}x^{n-1}$.
\end{itemize}
\end{lem}
\begin{proof}
We use induction on $n$. For $n=1$ the assertions are  trivial and indeed we have $\alpha_1(x)=x$ and $\beta_0(x)=1$. Next assume that the assertions hold for all $n'\leq n$. Suppose that $\alpha_n(x)=a'_0+\cdots+a'_n x^n$. Then,
\begin{align}
\label{eq:vpn++1}
\notag
v_{\mathfrak p}p^{n+1}
&=
\alpha_{n}(p)v_{\mathfrak p}p-v_{\mathfrak k}\beta_{n-1}(p)p
=
\alpha_{n}(p)(pv_{\mathfrak p}-v_{\mathfrak k})-v_{\mathfrak k}\beta_{n-1}(p)p\\
\notag
&
=\alpha_{n}(p)pv_{\mathfrak p}-\alpha_{n}(p)v_{\mathfrak k}-v_{\mathfrak k}\beta_{n-1}(p)p
=
\alpha_{n}(p)pv_{\mathfrak p}-\sum_{i=0}^{n}a'_ip^i v_{\mathfrak k}-v_{\mathfrak k}\beta_{n-1}(p)p
\\
\notag 
&
=\alpha_{n}(p)pv_{\mathfrak p}-\sum_{i=0}^{n}a'_i(v_{\mathfrak k}\alpha_{i}(p)+\beta_{i-1}(p)v_{\mathfrak p}
)-v_{\mathfrak k}\beta_{n-1}(p)p\\
&=\left(p\alpha_n(p)-\sum_{i=1}^{n}a'_i\beta_{i-1}(p)\right)v_{\mathfrak p}-v_{\mathfrak k}\left(p\beta_{n-1}(p)+\sum_{i=0}^na'_i\alpha_{i}(p)\right).
\end{align}
But by the hypothesis of induction and straightforward computations we also have 
\begin{align}
\label{eq:pn+1}
\notag
p^{n+1}v_{\mathfrak k}
&=
pv_{\mathfrak k}\alpha_{n}(p)+p\beta_{n-1}(p)v_{\mathfrak p}
=
(v_{\mathfrak k}p+v_{\mathfrak p})\alpha_{n}(p)+p\beta_{n-1}(p)v_{\mathfrak p}\\
\notag
&=v_{\mathfrak k}p\alpha_{n}(p)+v_{\mathfrak p}\alpha_{n}(p)+p\beta_{n-1}(p)v_{\mathfrak p}\\
&=v_{\mathfrak k}p\alpha_{n}(p)+\sum_{i=0}^na'_i\alpha_{i}(p)v_{\mathfrak p}-v_{\mathfrak k}\sum_{i=1}^na'_i\beta_{i-1}(p)+p\beta_{n-1}(p)v_{\mathfrak p}\\
&=v_{\mathfrak k}\left(p\alpha_{n}(p)-\sum_{i=1}^na'_i\beta_{i-1}(p)\right)+\left(p\beta_{n-1}(p)+\sum_{i=0}^na'_i\alpha_{i}(p)\right)v_{\mathfrak p}.\notag
\end{align}
Comparing the right hand sides of~\eqref{eq:pn+1} and~\eqref{eq:dfnab} 
yields
\[
\alpha_{n+1}(x) = x\alpha_{n}(x) - \sum_{i=1}^n a'_i \beta_{i-1}(x)
\quad\text{and}\quad
\beta_{n}(x) = x\beta_{n-1}(x) + \sum_{i=0}^n a'_i \alpha_i(x).
\]
This proves (iii) and (i). 
Finally, (ii) follows from comparing the right hand sides of~\eqref{eq:vpn++1} and~\eqref{eq:pn+1}. 
\end{proof}

\begin{lem}\label{alpha(x)}
For $n\geq 1$ we have \[
\beta_{n-1}(x)=\sum_{\tiny
\begin{array}{cc}
0\leqslant i\leqslant n\\
i \, \text{is odd}
\end{array}}{n \choose i}\alpha_{n-i}(x),\]
and 
\begin{align*}
\alpha_{n}(x)=\sum_{\tiny
\begin{array}{cc}
0\leqslant i\leqslant n\\
i \, \text{is even}
\end{array}} {n \choose i}x^{n-i}-\sum_{\tiny
\begin{array}{cc}
0\leqslant i,j\leqslant n\\
i,j \, \text{are odd}
\end{array}} {n \choose i} {n-i \choose j}\alpha_{n-i-j}(x).
\end{align*}
\end{lem}
\begin{proof}
 By Lemma \ref{supersymmetrygl(1,2)} and Lemma \ref{alpha(p)}(ii) we have
\begin{align*}
p^nv_{\mathfrak k}&
= \sum_{\tiny
\begin{array}{cc}
0\leqslant i\leqslant n\\
i \, \text{is even}
\end{array}}{n \choose i}
v_{\mathfrak k}
p^{n-i}+ \sum_{\tiny
\begin{array}{cc}
0\leqslant i\leqslant n\\
i \, \text{is odd}
\end{array}}{n \choose i}v_{\mathfrak p}p^{n-i}\\
&=\sum_{\tiny
\begin{array}{cc}
0\leqslant i\leqslant n\\
i \, \text{is even}
\end{array}}{n \choose i}
v_{\mathfrak k} 
p^{n-i}+ \sum_{\tiny
\begin{array}{cc}
0\leqslant i\leqslant n\\
i \, \text{is odd}
\end{array}}{n \choose i}(\alpha_{n-i}(p)v_{\mathfrak p}-v_{\mathfrak k}\beta_{n-i-1}(p))\\
&=v_{\mathfrak k} \left( \sum_{\tiny
\begin{array}{cc}
0\leqslant i\leqslant n\\
i \, \text{is even}
\end{array}}{n \choose i}p^{n-i}-\sum_{\tiny
\begin{array}{cc}
0\leqslant i\leqslant n\\
i \, \text{is odd}
\end{array}}{n \choose i}\beta_{n-i-1}(p))\right)+
\sum_{\tiny
\begin{array}{cc}
0\leqslant i\leqslant n\\
i \, \text{is odd}
\end{array}}{n \choose i}\alpha_{n-i}(p)v_{\mathfrak p}.
\end{align*}
The assertions of the lemma now follow from comparing the right hand side of the above calculation and the assumption $p^nv_{\mathfrak k}=v_{\mathfrak k}\alpha_{n}(p)+\beta_{n-1}(p)v_{\mathfrak p}$.
\end{proof}

\begin{prp}
\label{prp:alphaa}
We have $\alpha_{n}(x)=\sum_{k=0}^{\lfloor \frac n2\rfloor } E_{2k}{n \choose 2k}x^{n-2k}$ for $n\geq 1$.
\end{prp}
\begin{proof}
Using Lemma~\ref{alpha(x)}, it follows  from a simple induction on $n$ that the only powers of $x$ that occur in $\alpha_n(x)$ and in $\beta_n(x)$ are $x^{n-2k}$ for $k\in\Z^{\geq 0}$. Now suppose $\bar E_n\in\C$ are chosen such that 
\[
\alpha_{n}(x)=\sum_{k=0}^{\lfloor \frac n2 \rfloor} \bar{E}_{2k}{n \choose 2k}x^{n-2k}.
\]
Our goal is to prove that $\bar{E}_{2i}=E_{2i}$ for all $i\geq 0$.
From Lemma~\ref{alpha(x)} it follows that 
\begin{align*}
\sum_{k=0}^n \bar{E}_{2k}{n \choose 2k}x^{n-2k}=\sum_{\tiny
\begin{array}{cc}
0\leqslant i\leqslant n\\
i \, \text{is even}
\end{array}} {n \choose i}x^{n-i}-\sum_{\tiny
\begin{array}{cc}
0\leqslant i,j\leqslant n\\
i,j \, \text{are odd}
\end{array}} {n \choose i} {n-i \choose j}\alpha_{n-i-j}(x).
\end{align*}
We adopt the convention that \({\bar E}_k = 0\) for \(k < 0\). 
By comparing the coefficient of $x^{n-2k}$ on both sides we obtain
\begin{align*}
\bar{E}_{2k}{n \choose 2k}={n \choose 2k}-\sum_{\tiny
\begin{array}{cc}
0\leqslant i,j\leqslant n\\
i,j \, \text{are odd}
\end{array}} {n \choose i,j,n-2k,2k-i-j}\bar{E}_{2k-i-j}.
\end{align*}
We can also remove the constraint $i,j \leqslant n$ and write
\begin{align*}
\bar{E}_{2k}=1-\sum_{\tiny
\begin{array}{cc}
0\leqslant i,j\\
i,j \, \text{are odd}
\end{array}} {2k \choose i,j,2k-i-j}\bar{E}_{2k-i-j}.
\end{align*}
 Dividing both sides by $(2k)!$, we obtain
\begin{align*}
\frac{\bar{E}_{2k}}{(2k)!}=\frac{1}{(2k)!}-\sum_{\tiny
\begin{array}{cc}
0\leqslant i,j \\
i,j \, \text{are odd}
\end{array}} \frac{1}{i! j!}\frac{\bar{E}_{2k-i-j}}{(2k-i-j)!}.
\end{align*}
Then, by substituting $i=2A
+1$ and $j=2B+1$, we find that
\begin{align*}
\frac{\bar{E}_{2k}}{(2k)!}=\frac{1}{(2k)!}-\sum_{\tiny
\begin{array}{c}
0\leqslant A,B
\end{array}} \frac{1}{(2A+1)! (2B+1)!}\frac{\bar{E}_{2k-2A-2B-2}}{(2k-2A-2B-2)!}.
\end{align*}
Thus,  
\[
\sum_{k \geqslant 0}\frac{\bar{E}_{2k}}{(2k)!}x^{2k} = \sum_{k \geq 0}\frac{x^{2k}}{(2k)!} - \left(\sum_{A \geq 0} \frac{x^{2A+1}}{(2A+1)!}\right)\left(\sum_{B \geq 0} \frac{x^{2B+1}}{(2B+1)!}\right)\sum_{r \geqslant 0}\frac{\bar{E}_{2r}}{(2r)!}x^{2r}.
\]
But we have
\begin{align*}
\cosh(x)=\sum_{k \geq 0}\frac{x^{2k}}{(2k)!}, \quad \sinh(x)=\sum_{k \geq 0} \frac{x^{2k+1}}{(2k+1)!}.
\end{align*}
Therefore,
\[
\sum_{k \geqslant 0}\frac{\bar{E}_{2k}}{(2k)!}x^{2k} =
\left(\frac{e^{x} + e^{-x}}{2}\right) -
\left(\frac{e^{x} - e^{-x}}{2i}\right)^2 \sum_{r \geqslant 0}\frac{\bar{E}_{2r}}{(2r)!}x^{2r},
\]
which implies that
$
\sum_{ k \geqslant 0}\frac{\bar{E}_{2k}}{(2k)!}x^{2k}=
\frac{2}{e^{x}+e^{-x}}$.
The right hand side of the above equation is $\text{sech}(x)$. Since $\text{sech}(x)=\sec(ix)$, the Taylor series of $\text{sech}(x)$ is 
$\sum_{k \geqslant 0}\frac{E_{2k}}{(2k)!}x^{2k}$. 
This proves that  $\bar{E}_{2k}=E_{2k}$.
\end{proof}
\begin{prp}\label{beta(x)}
For $n\geq 1 $ we have \begin{align*}
\beta_{n-1}(x)=\sum_{k=0}^{\lfloor \frac{n-1}{2}\rfloor} \frac{2^{2k+2}(2^{2k+2}-1)B_{2k+2}}{2k+2}
{n \choose n-1-2k}x^{n-1-2k}.
\end{align*}
\end{prp}
\begin{proof}
The powers of $x$ that occur in $\beta_{n-1}(x)$ are $x^{n-1-2k}$ for $k\in\Z^{\geq 0}$. Assume that for some  $\bar{B}_i\in\C$ we have 
\begin{align*}
\beta_{n-1}(x)=\sum_{k=0}^{\lfloor \frac{n-1}{2}\rfloor} \frac{2^{2k+2}(2^{2k+2}-1)\bar{B}_{2k+2}}{2k+2}
{n \choose n-1-2k}x^{n-1-2k}.
\end{align*}
Then by Lemma \ref{alpha(x)} and Proposition~\ref{prp:alphaa} we obtain
\begin{align*}
\sum_{k=0}^{\lfloor \frac{n-1}{2}\rfloor} \frac{2^{2k+2}(2^{2k+2}-1)\bar{B}_{2k+2}}{2k+2}
{n \choose n-1-2k}x^{n-1-2k}&=\sum_{\tiny
\begin{array}{cc}
0\leqslant j\leqslant n\\
j \, \text{is odd}
\end{array}} {n \choose j}\alpha_{n-j}(x)\\
&=\sum_{\tiny
\begin{array}{cc}
0\leqslant j\leqslant n\\
j \, \text{is odd}
\end{array}}\sum_{r=0}^{n-j} {n \choose j}{n-j \choose 2r}E_{2r}x^{n-j-2r}.
\end{align*}
By comparing the coefficient of $x^{n-1-2k}$ on both sides we obtain
\begin{align*}
&\frac{2^{2k+2}(2^{2k+2}-1)\bar{B}_{2k+2}}{2k+2}
{n \choose n-1-2k}
=\sum_{\tiny
\begin{array}{cc}
0\leqslant j\leqslant n\\
j \, \text{is odd}
\end{array}}\sum_{k=j-1}^{n-j}{n \choose j,2k-j+1,n-2k-1}E_{2k+1-j}.
\end{align*}
Dividing both sides by $(2k+1)!{n\choose n-1-2k}$ and simplifying the binomial coefficient, we obtain
\begin{align*}
\frac{2^{2k+2}(2^{2k+2}-1)\bar{B}_{2k+2}}{(2k+2)!}=
\sum_{\tiny
\begin{array}{cc}
0\leqslant j\leqslant n\\
j \, \text{is odd}
\end{array}}\sum_{k=j-1}^{n-j}\frac{E_{2k+1-j}}{j!(2k+1-j)!}.
\end{align*}
Then, by the substitiution $j=2A
+1$ we obtain 
\begin{align*}
\sum_{k\geqslant 0}\frac{2^{2k+2}(2^{2k+2}-1)\bar{B}_{2k+2}}{(2k+2)!}x^{2k+1}=\sum_{A\geqslant 0}\frac{x^{2A+1}}{(2A+1)!}\sum_{r\geqslant 0}\frac{E_{2r}}{(2r)!}x^{2r}
\end{align*}
But we have
$\text{sech}(x)=\sum_{r\geqslant 0}\frac{E_{2r}}{(2r)!}x^{2r} $ and $\sinh(x)=\sum_{k \geq 0} \frac{x^{2k+1}}{(2k+1)!}$.
Thus,
\begin{align*}
\sum_{k\geqslant 0}\frac{2^{2k+2}(2^{2k+2}-1)\bar{B}_{2k+2}}{(2k+2)!}x^{2k+1}
=\frac{e^x-e^{-x}}{2}\frac{2}{e^x+e^{-x}}=
\frac{e^x-e^{-x}}{e^x+e^{-x}}.
\end{align*}
The right hand side of the above equation is the hyperbolic tangent function $\text{tanh}(x)$, with Taylor series 
\[
\text{tanh}(x)=\sum_{k\geqslant 0}\frac{2^{2k+2}(2^{2k+2}-1)B_{2k+2}}{(2k+2)!}x^{2k+1}.
\]
It follows immediately that $\bar{B}_{2k+2}=B_{2k+2}$.
\end{proof}

\begin{lem}\label{firstpart}
For every 
\[k_0=\begin{bmatrix}
0&0&0\\
0&\alpha&\beta\\
0&\gamma&-\alpha
\end{bmatrix}
\in \g k_{\bar{0}},
\] the following relations hold in $U(\g g)$:
\begin{itemize}
\item[\rm (i)] 
$p^nk_0=k_0p^n$.
\item[\rm (ii)]
$p^nek_0=k_0p^ne+p^n(\alpha e+\beta f)$.
\item[\rm (iii)] $p^nfk_0=k_0p^nf+p^n(\gamma e-\alpha f)$.
\item[\rm (iv)] $p^nefk_0=k_0p^nef+\beta k_1 p^n-\gamma k_2 p^n$.
\end{itemize}
\end{lem}
\begin{proof}
By straightforward calculations one can see that
\begin{align*}
 &p^nk_0=k_0p^n\in \g k U(\g g),\\
 &p^nek_0=p^nk_0e+p^n(\alpha e+\beta f)=k_0p^ne+p^n(\alpha e+\beta f),\\
 &p^nfk_0=p^nk_0f+p^n(\gamma e-\alpha f)=k_0p^nf+p^n(\gamma e-\alpha f),
\end{align*}
and
\begin{align*}
 p^nefk_0&=p^ne(k_0f+\gamma e-\alpha f)=p^nek_0f+\gamma p^ne^2-\alpha p^n ef\\
&=p^n(k_0e+(\alpha e+\beta f))f+\gamma p^ne^2-\alpha p^n ef\\
&=p^nk_0ef+\alpha p^nef+\beta p^nf^2+\gamma p^ne^2-\alpha p^n ef\\
&=k_0p^nef+\frac{1}{2}\beta p^n[f,f]+\frac{1}{2}\gamma p^n[e,e]=k_0p^nef+\beta p^nk_1-\gamma p^nk_2=k_0p^nef+\beta k_1 p^n-\gamma k_2 p^n.
\end{align*}
This completes the proof.\end{proof}

\begin{lem}\label{Second part}
Let $v=(a,b)\in\C^2$. Then the following relations hold in $U(\g g)$:
\begin{itemize}
\item[\rm (i)]
$p^nv_{\mathfrak k}=v_{\mathfrak k}\alpha_{n}(p)+\beta_{n-1}(p)v_{\mathfrak p}$.
\item[\rm (ii)]
$p^nev_{\mathfrak k}=-v_{\mathfrak k}\alpha_{n}(p)e+ak_2\beta_{n-1}(p)-bk\beta_{n-1}(p)+b\beta_{n-1}(p)ef-bp^{n+1}$.

\item[\rm (iii)]
$ p^nfv_{\mathfrak k}=-v_{\mathfrak k}\alpha_{n}(p)f-a\beta_{n-1}(p)ef-bk_1\beta_{n-1}(p)+ap^{n+1}$.
\item[\rm (iv)]
$
p^nefv_{\mathfrak k}=D-bk_1\beta_{n-1}(p)e-b\beta_{n-1}(p)f+p^{n+1}v_{\mathfrak p}-ae'\alpha_{n}(p)-\beta_{n-1}(p)e$, 
where \[
D:=v_{\mathfrak k}\alpha_{n}(p)ef-ak_2\beta_{n-1}(p)+bk\beta_{n-1}(p)f.
\]
\end{itemize}
\end{lem}
\begin{proof}
Part (i) is just the defining relation~\eqref{eq:dfnab}. By straightforward calculations we have
\begin{align*}
[e,v_{\mathfrak k}]&=-bp, \quad [f,v_{\mathfrak k}]=ap,\quad [e,e]=-2k_2,\quad [f,f]=2k_1,\quad [e,p]=-e',\quad [e',p]=-e,\\
[e',e']&=2k_2,\quad [f',f']=-2k_1,\quad [p,k_1]=[p,k_2]=[p,k]=0,\quad [e,k_1]=f.
\end{align*}
Therefore, by Lemmas \ref{alpha(x)} and \ref{beta(x)} we obtain
\begin{align*}
p^nev_{\mathfrak k}
&=p^n(-v_{\mathfrak k}e-bp)=-p^nv_{\mathfrak k}e-bp^{n+1}=-v_{\mathfrak k}\alpha_{n}(p)e-\beta_{n-1}(p)v_{\mathfrak p}e-bp^{n+1}\\
&=-v_{\mathfrak k}\alpha_{n}(p)e-\beta_{n-1}(p)(ae+bf)e-bp^{n+1}\\
&=-v_{\mathfrak k}\alpha_{n}(p)e-a\beta_{n-1}(p)e^2-b\beta_{n-1}(p)(k-ef)-bp^{n+1}\\
&=-v_{\mathfrak k}\alpha_{n}(p)e+a\beta_{n-1}(p)k_2-b\beta_{n-1}(p)k+b\beta_{n-1}(p)ef-bp^{n+1}\\
&=-v_{\mathfrak k}\alpha_{n}(p)e+ak_2\beta_{n-1}(p)-bk\beta_{n-1}(p)+b\beta_{n-1}(p)ef-bp^{n+1}.
\end{align*}
This proves (ii). Similarly, 
\begin{align*}
  p^nfv_{\mathfrak k}&
=p^n(-v_{\mathfrak k}f+ap)=-p^nv_{\mathfrak k}f+a
p^{n+1}=-v_{\mathfrak k}\alpha_{n}(p)f-\beta_{n-1}(p)v_{\mathfrak p}f+ap^{n+1}\\
&=-v_{\mathfrak k}\alpha_{n}(p)f-\beta_{n-1}(p)(ae+bf)f+ap^{n+1}\\
&=-v_{\mathfrak k}\alpha_{n}(p)f-a\beta_{n-1}(p)ef-b\beta_{n-1}(p)f^2+ap^{n+1}\\
&=-v_{\mathfrak k}\alpha_{n}(p)f-a\beta_{n-1}(p)ef-b\beta_{n-1}(p)k_1+ap^{n+1}\\
&=-v_{\mathfrak k}\alpha_{n}(p)f-a\beta_{n-1}(p)ef-bk_1\beta_{n-1}(p)+ap^{n+1}.
\end{align*}
This proves (iii). For (iv) first note that 
\begin{align}
\label{eq:soim}
\notag
 p^nefv_{\mathfrak k}&=p^n\left(v_{\mathfrak k}ef+e[v_{\mathfrak k},f]-[v_{\mathfrak k},e]f\right)=p^nv_{\mathfrak k}ef+p^ne(ap)+bp^{n+1}f\\
&=p^nv_{\mathfrak k}ef+ap^n(pe-e')+bp^{n+1}f=p^nv_{\mathfrak k}ef+ap^{n+1}e-ap^ne'+bp^{n+1}f\\
&=p^nv_{\mathfrak k}ef+p^{n+1}v_{\mathfrak p}-ap^ne'.
\notag
\end{align}
From~\eqref{eq:dfnab} we also have
$ap^ne'=ae'\alpha_{n}(p)+\beta_{n-1}(p)e$.
Thus, the right hand side of~\eqref{eq:soim} can be simplified accordingly, as follows: 
\begin{align*}
p^nefv_{\mathfrak k}&=v_{\mathfrak k}\alpha_{n}(p)ef+\beta_{n-1}(p)v_{\mathfrak p}ef+p^{n+1}v_{\mathfrak p}-ae'\alpha_{n}(p)-\beta_{n-1}(p)e\\
&=v_{\mathfrak k}\alpha_{n}(p)ef+\beta_{n-1}(p)(ae+bf)ef+p^{n+1}v_{\mathfrak p}-ae'\alpha_{n}(p)-\beta_{n-1}(p)e\\
&=v_{\mathfrak k}\alpha_{n}(p)ef+a\beta_{n-1}(p)e^2+b\beta_{n-1}(p)fef+p^{n+1}v_{\mathfrak p}-ae'\alpha_{n}(p)-\beta_{n-1}(p)e\\
&=v_{\mathfrak k}\alpha_{n}(p)ef-a\beta_{n-1}(p)k_2+b\beta_{n-1}(p)(k-ef)f+p^{n+1}v_{\mathfrak p}-ae'\alpha_{n}(p)-\beta_{n-1}(p)e\\
&=D-b\beta_{n-1}(p)ef^2+p^{n+1}v_{\mathfrak p}-ae'\alpha_{n}(p)-\beta_{n-1}(p)e\\
&=D-b\beta_{n-1}(p)ek_1+p^{n+1}v_{\mathfrak p}-ae'\alpha_{n}(p)-\beta_{n-1}(p)e\\
&=D-b\beta_{n-1}(p)(k_1e+f)+p^{n+1}v_{\mathfrak p}-ae'\alpha_{n}(p)-\beta_{n-1}(p)e\\
&=D-bk_1\beta_{n-1}(p)e-b\beta_{n-1}(p)f+p^{n+1}v_{\mathfrak p}-ae'\alpha_{n}(p)-\beta_{n-1}(p)e.\qedhere
\end{align*}
\end{proof}
\begin{theorem}\label{basis of I}
The following vectors in $U(\g g)$ form  a basis of $\cI$:

\begin{itemize}
\item[\rm (i)] $z^mp^ne$ and $z^mp^nf$ where $m,n\geqslant 0$,
\item[\rm (ii)] $z^m\beta_{n-1}(p)ef-z^mp^{n+1}$ where $m,n\geqslant 0$, with the convention $\beta_{-1}(x):=0$,\item[\rm (iii)] $z^mk^{r_1}k_1^{r_2}k_2^{r_3}(e')^{r_4}(f')^{r_5}p^{s_1}e^{s_2}f^{s_3}$ where $m,r_1,r_2,r_3,s_1\geq 0$, $r_4,r_5,s_2,s_3\in\{0,1\}$, and 
 at least one of $r_1,\cdots , r_5\geq 1$.
\end{itemize}\end{theorem}

\begin{proof}
We remark that $z$ is in the centre of $U(\g g)$.
Let $\mathcal B_\mathrm{(i)},\mathcal B_\mathrm{(ii)}$
and 
$\mathcal B_\mathrm{(iii)}$ denote the sets of vectors defined in (i)-(iii) above, respectively. Then from Lemma~\ref{firstpart}(ii) it follows that $\mathcal B_\mathrm{(i)}\sseq \cI$. From Lemma~\ref{Second part}(iii) it follows that $\mathcal B_\mathrm{(ii)}\sseq \cI$. It is also clear that $\mathcal B_\mathrm{(iii)}\sseq\cI$.  

To prove that $\mathcal B_\mathrm{(i)}\cup\mathcal B_\mathrm{(ii)}\cup\mathcal B_\mathrm{(iii)}$ spans $\cI$, note that elements of $U(\g g)\g k$ that are not in $\g kU(\g g)$ are linear combinations of the
 monomials in the PBW basis that, modulo a factor of $z^m$,  appear on  the left hand sides of the relations in Lemmas~\ref{firstpart} and~\ref{Second part}. 
From these lemmas it follows that the latter monomials are in the span of $\mathcal B_\mathrm{(i)}\cup\mathcal B_\mathrm{(ii)}\cup\mathcal B_\mathrm{(iii)}$.

Finally, we prove that the vectors
in $\mathcal B_\mathrm{(i)}\cup\mathcal B_\mathrm{(ii)}\cup\mathcal B_\mathrm{(iii)}$
 are linearly independent. Consider any  linear dependence relation between the elements of 
$\mathcal B_\mathrm{(i)}\cup\mathcal B_\mathrm{(ii)}\cup\mathcal B_\mathrm{(iii)}$.  
By organizing the dependence relation according to the powers of $z$, and factoring out the common power $z^m$, we can assume that in all of the monomials $m=0$. 
  Let $c_n$ denote the coefficient of $\beta_{n-1}(p)ef-p^{n+1}$ for $n\geq 0$, and set $N:=\max\{n:c_n\neq 0\}$. Recall that $\deg(\beta_{N-1})=N-1$. Now with the PBW basis of $U(\g g)$ in mind, note that the summand $c_Np^{N-1}ef$ cannot get canceled in the linear dependence relation. This is a contradiction. 
\end{proof}

\bibliographystyle{plain}
\bibliography{References.bib}

\end{document}